\documentclass[UTF-8,reqno,12pt]{amsart}
\usepackage{enumerate,amsmath,amssymb,amsthm}
\usepackage{amssymb,url,color}
\usepackage{hyperref}
\usepackage{amsmath,amsthm}
\usepackage{mathrsfs}
\usepackage[notref,notcite]{showkeys}
\usepackage{bm}
\usepackage{graphicx}
\usepackage{epstopdf}
\usepackage{marvosym}
\allowdisplaybreaks[4]

\newtheorem{thm}{Theorem}[section]
\theoremstyle{Condition}

\newtheorem{lem}{Lemma}[section]
\newtheorem{rem}{Remark}[section]
\newtheorem{exa}{Example}[section]
\theoremstyle{Problem}

\theoremstyle{Assumption}

\theoremstyle{Definition}

\numberwithin{equation}{section}

\def\beq{\begin{equation}}
\def\deq{\end{equation}}
\allowdisplaybreaks 

\bfseries\rmfamily 
\def\cA{{\mathcal A}}
\def\cB{{\mathcal B}}
\def\cC{{\mathcal C}}

\def\cF{{\mathcal F}}

\def\cL{{\mathcal L}}

\def\cP{{\mathcal P}}

\def\cT{{\mathcal T}}

\def\cW{{\mathcal W}}

\def\mE{{\mathbb E}}

\def\mN{{\mathbb N}}

\def\mP{{\mathbb P}}

\def\mR{{\mathbb R}}

\def\geq{\geqslant}
\def\leq{\leqslant}

\def\a{\alpha}

\def\b{\beta}

\def\d{\delta}

\def\s{\sigma}

\def\[{{\Big[}}
\def\]{{\Big]}}
\def\<{{\langle}}
\def\>{{\rangle}}
\def\({{\Big(}}
\def\){{\Big)}}

\def\dif{{\rm d}}

\def\={&\!\!=\!\!&}
\def\bt{\begin{theorem}}
\def\et{\end{theorem}}
\def\bl{\begin{lemma}}
\def\el{\end{lemma}}
\def\br{\begin{rem}}
\def\er{\end{rem}}

\begin{document}

\title[EM method for invariant measures of McKean-Vlasov SDEs]
{The Euler-Maruyama method for invariant measures of McKean-Vlasov stochastic differential equations}

\author{Zhen Wang$^{1}$, Mingyan Wu$^{2}$}

\thanks{\Letter \ \  Zhen Wang\\
  \ \  E-mail addresses$:$ wangzhen881025@163.com}
\thanks{ \ \  Mingyan Wu\\
  \ \  E-mail addresses$:$ mingyanwu@hust.edu.cn ; mingyanwu.math@gmail.com}
\thanks{$^{1}$ School of Mathematics and Statistics (School of Grytology), Henan Normal University, Xinxiang, Henan, $453007$, P. R. China. }
\thanks{$^{2}$ School of Mathematical sciences, Xiamen University, Xiamen, Fujian, $361005$, P. R. China. }
\thanks{This work is partially supported by National Natural Science Foundation of China (No.12501192).}
\thanks{This work is partially supported by Natural Science Foundation of Henan Province in China (No.242300420646).}

\begin{abstract}
\ \ This paper investigates the approximation of invariant measures for McKean-Vlasov stochastic differential equations (SDEs) using the Euler-Maruyama (EM) scheme under a monotonicity condition. Firstly, the convergence of the numerical solution from the EM scheme to its continuous-time counterpart is established. Secondly, we show that the numerical solution admits a unique invariant measure and derive its convergence rate under the Wasserstein metric. In parallel, it is demonstrated that the associated particle system also possesses these properties.
\end{abstract}

\maketitle

\section{Introduction}
  A McKean-Vlasov process is defined by stochastic differential equations whose coefficients depend on the law of the solution itself, providing a mathematical formulation of the propagation of chaos in mean-field interacting particle systems (cf. \cite{K,M,S,SKP}). This framework includes important models such as the Vlasov-Fokker-Planck equation \cite{P} and various mean-field interacting systems \cite{RF, RF2}.

  The study of invariant measures is fundamental to understanding the long-time behavior of stochastic processes, and is particularly relevant in the context of McKean-Vlasov SDEs.
 Under a monotonicity condition, Wang \cite{W} established the unique invariant measure for McKean-Vlasov SDEs and proved exponential convergence. Subsequently,
the weak well-posedness, along with the existence and uniqueness of an invariant probability measure, was investigated for McKean-Vlasov SDEs with integrable drift in \cite{HWY}. In a related direction, the authors of \cite{BSY} demonstrated the existence of an invariant probability measure for a class of functional
 McKean-Vlasov SDEs employing Kakutani's fixed point theorem to an appropriate space of probability measures on continuous functions. Using non-symmetric singular granular media equations, \cite{W1} studied exponential ergodicity in relative entropy and (weighted) Wasserstein distances for reflecting McKean-Vlasov SDEs. Finally, Zhang \cite{Zh} provided conditions under which McKean-Vlasov SDEs may admit more than one invariant measure.

The approximation of invariant measures for McKean-Vlasov SDEs has been explored from the perspective of the process itself.  Specifically, Du et al. \cite{DJL} demonstrated that the empirical measures of the solution process converge to its invariant measure and quantified the rate in terms of the Wasserstein distance. Subsequently, Cao and Du \cite{CD} further explored this problem for dynamics with non-degenerate additive noise, establishing convergence results in the Wasserstein metric.

The general intractability of analytical solutions and invariant measures for McKean-Vlasov SDEs in practical settings compels the adoption of numerical methods. This paper investigates the approximation of the invariant measure of McKean-Vlasov SDEs via the EM scheme under a monotonicity condition.

Considering the following McKean-Vlasov SDE:
\beq\label{a}
X_t=X_0+\int_0^tb(X_s,\cL_{X_s})\dif s+\int_0^t\s(X_s,\cL_{X_s})\dif W_s, \ \ t\in[0,T],
\deq
with the initial value $X_0=x$. Here, $\cL_{X_t}$ denotes the law of $X_t$ and $W_t$ is a standard $d$-dimensional Wiener process defined on a complete filtered probability space $(\Omega, \cF,\cF_t,\mP)$.

To investigate the EM scheme of the invariant measure for McKean-Vlasov SDEs, we introduce an equidistant temporal partition. For a step size $h\in(0,1)$, the grid is defined by $\mathcal{T}^h= \{kh : k \in \mN\}$, which covers the finite time horizon $[0,T]$.  The EM scheme for simulating the solution to (\ref{a}) is then given by the recursive relation on this grid $\mathcal{T}^h$:
\beq\label{a-2}
\hat{X}_{(k+1)h}=\hat{X}_{kh}+hb(\hat{X}_{kh},\cL_{\hat{X}_{kh}})+\s(\hat{X}_{kh},\cL_{\hat{X}_{kh}})\Delta W_{(k+1)h},
\deq
where the initial value $\hat{X}_0=x$ and $\Delta W_{(k+1)h}:=W_{(k+1)h}-W_{kh}$ is the Brownian increment, satisfying
$\mE[|\Delta W_{(k+1)h}|^2]=h$.

Furthermore, we consider a system of $N$ interacting particles that approximates \eqref{a} in the mean-field sense. For $j = 1,2,\dots,N$, the dynamics are given by
\beq\label{b}
 X_{t}^{N,j}=X_0^{N,j}+\int_0^tb\left(X_{s}^{N,j},\mu_{s}^N\right)\dif s+\int_0^t\s\left(X_{s}^{N,j},\mu_{s}^N\right)\dif W_s^j,\ \ X_0^{N,j}=X_0,
\deq
where $\mu_t^N$ is the empirical measure of $\left\{X_t^{N,i},i=1,2,\cdots,N\right\}$, defined as
$$
\mu_t^N:=\frac{1}{N}\sum_{i=1}^N\d_{X_t^{N,i}},
$$
 with $\d_x$ the Dirac measure at $x$. Subsequently, the EM scheme to $N$ interacting particles (\ref{b}) on the grid $\mathcal{T}^h$ is:
\beq\label{b-2}
\hat{X}_{(k+1)h}^{N,j,h}=\hat{X}_{kh}^{N,j,h}+hb\left(\hat{X}_{kh}^{N,j,h},\frac{1}{N}\sum_{i=1}^N\d_{\hat{X}_{kh}^{N,i,h}}\right)
+\s\left(\hat{X}_{kh}^{N,j,h},\frac{1}{N}\sum_{i=1}^N\d_{\hat{X}_{kh}^{N,i,h}}\right)\Delta W_{(k+1)h}.
\deq
As the number of particles $N$ tends to infinity, the interacting particle system exhibits propagation of chaos. Consequently, the discretized system's behavior converges to that of a non-interacting particle system, governed by the following equations:
\beq\label{c}
 X_{t}^{j}=X_0^j+\int_0^tb\left(X_{s}^{j},\cL_{X_{s}^{j}}\right)\dif s+\s\left(X_{s}^{j},\cL_{X_{s}^{j}}\right)\dif W_s^j,\ \ X_0^j=X_0.
\deq
Since the particles $X_{t}^j$ are independent, it follows that $\cL_{X_{t}^j}=\cL_{X_{t}}$, for every $j=1,2,\cdots,N$. The corresponding EM scheme for (\ref{c}) is given by:
\beq\label{c-2}
 \hat{X}_{(k+1)h}^{j,h}=\hat{X}_{kh}^{j,h}+hb\left(\hat{X}_{kh}^{j,h},\cL_{\hat{X}_{kh}^{j,h}}\right)+\s\left(\hat{X}_{kh}^{j,h},\cL_{\hat{X}_{kh}^{j,h}}\right)\Delta W_{(k+1)h}.
\deq

This work establishes the convergence of the EM scheme for approximating invariant measures of McKean-Vlasov SDEs under a monotonicity condition. We prove that the EM scheme converges to its continuous-time counterpart and admits a unique invariant measure at a quantitative rate in the Wasserstein distance. These results are further extended to the associated interacting particle system.

The paper is organized as follows. Section 2 introduces the mathematical preliminaries and states the main results. Section 3 is devoted to two key results: the convergence of the EM scheme to its continuous-time counterpart, and the existence and uniqueness of an invariant measure for the numerical solution. Section 4 extends the analysis to the interacting particle system, providing a detailed study of its properties.  Finally, numerical simulations are displayed in Section 5.
\section{Mathematical preliminaries and main results}
\subsection{Mathematical preliminaries}
In this section, we introduce notations and mathematical preliminaries used throughout the paper.

Let $|\cdot|$ denote the Frobenius norm on matrix spaces and $\cB(\mR^d)$ the family of Borel sets in $\mR^d$. For real numbers $a$ and $b$, $a \vee b$ and $a \wedge b$ denote the larger and smaller values, respectively. The inequality $a \lesssim b$ means that there exists a universal constant $C > 0$, independent of the key parameters, such that $a \leq C b$. For a matrix $B$, $\|B\|_{\mathrm{HS}}$ means its Hilbert-schmidt norm. Let $\cP(\mR^d)$  denote the space of probability measures on $\mR^d$, equipped with the 2-Wasserstein distance $\cW_2$. Within this space, $\cP_p$ represents the subset of measures with finite $p$-th moment, meaning for any $\mu \in \cP_p(\mR^d)$, we have $\mu(|\cdot|^p) := \int_{\mR^d} |x|^p \mu(dx) < \infty$.

For any $\mu, \nu \in \cP_p(\mR^d)$ with $p > 0$, the $\cW_p$-Wasserstein distance is defined by
\beq\label{2-a}
\cW_p(\mu,\nu)=\inf_{\pi\in\cC(\mu,\nu)}\left(\int_{\mR^d\times \mR^d}|x-y|^p\pi(\dif x,\dif y)\right)^{\frac{1}{1\vee p}},
\deq
where $\cC(\mu, \nu)$ is the set of all couplings of $\mu$ and $\nu$. A measure $\pi$ on $\mR^d \times \mR^d$ belongs to $\cC(\mu, \nu)$ if and only if $\pi( \cdot, \mR^d) = \mu(\cdot)$ and $\pi(\mR^d, \cdot) = \nu(\cdot)$.

For any $t\in\mR_+$ and any Borel set $B\in\cB(\mR^d)$, we denote by $\mP_{t}(x, B)$ the transition probability kernel of the process $X_{t}$. A probability measure $\pi \in \cP(\mR^d)$ is called an invariant measure for $X_{t}$ if it satisfies
\beq\label{3}
\pi(B)=\int_{\mR^d}\mP_{t}(x,B)\pi(\dif x),
\deq
 for any $B\in\cB(\mR^d)$.

Finally, we denote by $C$ (with or without subscripts) a generic positive constant whose value may vary from line to line and depends only on the indicated parameters.
\subsection{Main results}
Regarding the coefficients $b$ and $\s$, we make the following assumptions:\\
$\mathbf{Assumption}$ $\mathbf{1.}$ The coefficients $b(\cdot,\cdot)$ and $\s(\cdot,\cdot)$ are continuous, respectively defined on $\mR^d\times\cP_2$. \\
$\mathbf{Assumption}$ $\mathbf{2.}$ There are constants $\a>\beta\geq 0$, $\gamma>\kappa\geq0$, and $\rho>0$ such that for all $x,y\in\mR^d$ and $\mu,\nu\in\cP_2$,
$$
2\langle b(x,\mu)-b(y,\nu),x-y\rangle+\|\s(x,\mu)-\s(y,\nu)\|_{\mathrm{HS}}^2\leq -\a|x-y|^2+\beta \cW_2(\mu,\nu)^2
$$
and
$$
2\langle b(x,\mu),x\rangle+(1+\rho)\|\s(x,\mu)\|_{\mathrm{HS}}^2\leq -\gamma|x|^2+\kappa(1+\rho+\mu(|\cdot|^2)).
$$
$\mathbf{Assumption}$ $\mathbf{3.}$ For any $x,y\in\mR^d$ and $\mu,\nu\in\cP_2$, there exist positive constants $c_0$, $a$ and $b$ such that
$$
| b(x,\mu)|^2\vee \|\s(x,\mu)\|_{\mathrm{HS}}^2\leq c_0(1+|x|^2+\mu(|\cdot|^2))
$$
and
$$
|b(x,\mu)-b(y,\nu)|^2\vee \|\s(x,\mu)-\s(y,\nu)\|_{\mathrm{HS}}^2\leq a|x-y|^2+b \cW_2(\mu,\nu)^2.
$$
$\mathbf{Assumption}$ $\mathbf{4.}$ The constants satisfy $a+b<\alpha-\beta$.
\begin{rem}Under the above assumptions, it is established in \cite{W} that the solution to equation (\ref{a}) admits a unique invariant measure. Moreover, the solution converges exponentially to this invariant measure in the Wasserstein metric.
\end{rem}
\begin{rem} $\mathbf{Assumption}$ $\mathbf{2}$ provides a sufficient but not necessary condition for the existence and uniqueness of invariant measures for (\ref{a}). Nevertheless, it is sharp in the sense that the condition $\a>\beta$ cannot be relaxed. This is supported by counterexamples provided in \cite{DJL} and \cite{Zh}.
\end{rem}

Based on the above assumptions, we establish the following convergence results for the EM scheme, which is consistent with the classical theory for SDEs (see \cite{LMW, We}).
\begin{thm}\label{thm-1}Suppose that $\mathbf{Assumptions}$ $\mathbf{1}$-$\mathbf{3}$ hold and that $\mE|X_0|^2<\infty$. Then the EM scheme (\ref{a-2}) for (\ref{a}) satisfies
$$
\sup_{k=1,2,\cdots,\lfloor\frac{T}{h}\rfloor}\mE\left|X_{kh}-\hat{X}_{kh}\right|^2\preceq h.
$$
\end{thm}

\begin{thm}\label{thm-2} Let $\mathbf{Assumptions}$ $\mathbf{1}$-$\mathbf{4}$ hold. Define $h^{\sharp} := h^* \wedge h^{**}$. Then for any step size $h \in (0, h^{\sharp})$, the following assertions hold:\\
$\mathbf{(1)}$ The EM scheme (\ref{a-2}) admits a unique invariant measure $\hat{\pi}_h \in \cP(\mR^d)$.  Moreover, the numerical solution converges exponentially to $\hat{\pi}_h$ in the Wasserstein distance on $\cT^h$.\\
$\mathbf{(2)}$ The approximation error between the numerical invariant measure and the true invariant measure satisfies
    $$
    \cW_2(\hat{\pi}_h, \pi) \lesssim h^{\frac{1}{2}},
    $$
    where the implicit constant is independent of $h$ and $\pi$ is the unique invariant measure for (\ref{a}).
\end{thm}
Furthermore, for the interacting particle system, we derive an analogous result, as follows:
\begin{thm}\label{thm-3} Under $\mathbf{Assumptions}$ $\mathbf{1}$-$\mathbf{3}$ and given $\mE|X_0|^2 < \infty$, we establish a relationship between the EM scheme (\ref{b-2}) for the $N$-interacting particle system and its non-interacting counterpart (\ref{c-2}):
$$
\sup_{k=1,2,\cdots,\lfloor\frac{T}{h}\rfloor,j=1,2,3\cdots,N}\mE\left|\hat{X}_{kh}^{N,j,h}-\hat{X}_{kh}^{j,h}\right|^2\preceq \frac{h}{N}.
$$
Hence, we can arrive at
$$
\lim_{N\rightarrow\infty}\sup_{k=1,2,\cdots,\lfloor\frac{T}{h}\rfloor,j=1,2,3\cdots,N}\mE\left|\hat{X}_{kh}^{N,j,h}-\hat{X}_{kh}^{j,h}\right|^2=0.
$$
\end{thm}
\begin{rem}Theorem \ref{thm-3} shows the interplay between the discrete parameters: for a fixed $h$, the convergence rate between the $N$-interacting and non-interacting particle systems is $O (1/\sqrt{N})$ in Wasserstein distance; for a fixed $N$, the convergence rate between EM scheme (\ref{a-2}) and equation (\ref{a}) is $O (h^{1/2})$.
\end{rem}
\begin{rem}For any $t\in[0,T]$, the discrete-time error estimate can be extended to the continuous-time setting. Specifically, by a standard argument using the Lipschitz continuity of the coefficients and the Burkholder-Davis-Gundy inequality, we obtain:
$$
\mE\left[\sup_{t\in[0,\lfloor\frac{T}{h}\rfloor h]}|\hat{X}_{t}^{N,j,h}-\hat{X}_{t}^{j,h}|^2\right]\preceq \mE\left[\sum_{k=0}^{\lfloor\frac{T}{h}\rfloor}|\hat{X}_{kh}^{N,j,h}-\hat{X}_{kh}^{j,h}|^2\right]
+\frac{T}{N}.$$
Applying Gronwall's inequality to the corresponding integral inequality yields:
$$
\mE\left[\sup_{t\in[0,T]}|\hat{X}_{t}^{N,j,h}-\hat{X}_{t}^{j,h}|^2\right]\preceq \int_0^T\mE\left|\hat{X}_{t}^{N,j,h}-\hat{X}_{t}^{j,h}\right|^2\dif t+\frac{1}{N}.
$$
A further application of Gronwall's inequality then implies
$$
\sup_{j=1,2,3\cdots,N}\mE\left[\sup_{t\in[0,T]}|\hat{X}_{t}^{N,j,h}-\hat{X}_{t}^{j,h}|^2\right]\preceq \frac{1}{N}.
$$
This result is consistent with the continuous-time analysis in \cite{HRW} and demonstrates that the particle system approximation error remains $O(\frac{1}{N})$ even in the supremum norm over the time interval.
\end{rem}
\begin{thm}\label{thm-4} Let $\mathbf{Assumptions}$ $\mathbf{1}$-$\mathbf{4}$ hold and assume $\mE|X_0|^2 < \infty$. Set $h^{\sharp} = h^* \wedge h^{**}$. Then for any step size $h \in (0, h^{\sharp})$, the following holds:\\
$\mathbf{(1)}$ The EM scheme (\ref{b-2}) for the $N$-interacting particle system (\ref{c-2}) admits a unique invariant measure $\hat{\pi}_h^N$, and the numerical solution converges to $\hat{\pi}_h^N$ in the Wasserstein distance.\\
$\mathbf{(2)}$ The numerical invariant measure $\hat{\pi}_h^N$ approximates the true invariant measure $\pi$ of the non-interacting system. More precisely,
$$
\lim_{N\rightarrow\infty}\cW_2(\hat{\pi}_h^N, \pi)\preceq h^{\frac{1}{2}}.
$$
\end{thm}
\section{The numerical invariant measure}
The primary objective of this section is to demonstrate the convergence of the EM scheme (\ref{a-2}) to its continuous-time counterpart (\ref{a}) and to establish the existence and uniqueness of an invariant measure for the numerical solution.
\subsection{The properties of numerical invariant measure}
In this subsection, we present several key lemmas and provide the proof of Theorem \ref{thm-1}.
\begin{lem}\label{lem-1}Assume that $\mathbf{Assumptions}$ $\mathbf{2}$-$\mathbf{3}$ hold and that $\mE|X_0|^2<\infty$. Then there exists $h^*\in(0,1)$,  such that for any step size $h\in(0,h^*)$, the solution produced by the EM scheme (\ref{a-2}) satisfies
$$
\sup_{k=1,2,\cdots,\lfloor\frac{T}{h}\rfloor}\mE|\hat{X}_{kh}|^2\leq C_1,
$$
where the constant $C_1$ is independent of $k$ and $h$.
\end{lem}
\begin{proof}For $k=1,2,\cdots,\lfloor\frac{T}{h}\rfloor$, squaring both sides of (\ref{a-2}) gives
$$
\aligned
|\hat{X}_{(k+1)h}|^2=&|\hat{X}_{kh}|^2+|hb(\hat{X}_{kh},\cL_{\hat{X}_{kh}})|^2+|\s(\hat{X}_{kh},\cL_{\hat{X}_{kh}})\Delta W_{(k+1)h}|^2\\
&+2\langle \hat{X}_{kh},hb(\hat{X}_{kh},\cL_{\hat{X}_{kh}})\rangle+2\langle\hat{X}_{kh},\s(\hat{X}_{kh},\cL_{\hat{X}_{kh}})\Delta W_{(k+1)h}\rangle\\
&+2\langle hb(\hat{X}_{kh},\cL_{\hat{X}_{kh}}),\s(\hat{X}_{kh},\cL_{\hat{X}_{kh}})\Delta W_{(k+1)h}\rangle.\\
\endaligned
$$
Since $\Delta W_{(k+1)h}$ is independent of $\cF_{kh}$ and $\mE[\Delta W_{(k+1)h}]=0$, we have
$$
\mE[\langle\hat{X}_{kh},\s(\hat{X}_{kh},\cL_{\hat{X}_{kh}})\Delta W_{(k+1)h}\rangle]=0,$$ $$ \mE[\langle b(\hat{X}_{kh},\cL_{\hat{X}_{kh}}),\s(\hat{X}_{kh},\cL_{\hat{X}_{kh}})\Delta W_{(k+1)h}\rangle]=0.
$$
Taking expectations and applying $\mathbf{Assumptions}$ $\mathbf{2}$-$\mathbf{3}$, we obtain that
$$
\aligned
\mE|\hat{X}_{(k+1)h}|^2=&\mE|\hat{X}_{kh}|^2+h^2\mE|b(\hat{X}_{kh},\cL_{\hat{X}_{kh}})|^2+h\mE\|\s(\hat{X}_{kh},\cL_{\hat{X}_{kh}})\|_{\mathrm{HS}}^2
+2h\mE[\langle \hat{X}_{kh},b(\hat{X}_{kh},\cL_{\hat{X}_{kh}})]\\
\leq &
\mE[c_0h^2(1+|\hat{X}_{kh}|^2+\mu(|\hat{X}_{kh}|^2))+h(-\gamma|\hat{X}_{kh}|^2+\kappa(1+\rho+\mu(|\hat{X}_{kh}|^2)))]\\
&+\mE|\hat{X}_{kh}|^2\\
\leq &\mE[(1-\gamma h+c_0h^2)|\hat{X}_{kh}|^2+(c_0h^2+\kappa h)\mu(|\hat{X}_{kh}|^2)+c_0 h^2+h\kappa(1+\rho)]\\
=:&A_1\mE|\hat{X}_{kh}|^2+A_2,\\
\endaligned
$$
where
 $A_1:=1-(\gamma-\kappa) h+2c_0h^2$ and $A_2=:c_0 h^2+h\kappa(1+\rho)$.
Choose $h^*>0$ sufficiently small such that $A_1 \in (0,1)$ for all $h \in (0, h^*)$. Then by iteration,
$$
\mE|\hat{X}_{(k+1)h}|^2\leq A_1^k\mE|\hat{X}_0|^2+\frac{A_2(1-A_1^k)}{1-A_1}\leq \mE|X_0|^2+\frac{A_2}{1-A_1}:=C_1.
$$
Note that $\frac{A_2}{1-A_1}\leq\frac{c_0h^*+\kappa(1+\rho)}{\gamma-\kappa-2c_0h^*}$. Therefore, for all $h\in(0,h^*)$, $C_1$ is a well-defined positive constant. This completes the proof.
\end{proof}
\begin{lem}\label{lem-2} Under $\mathbf{Assumptions}$ $\mathbf{2}$-$\mathbf{4}$, then there exists $h^{**}\in(0,1)$ such that for any step size $h\in(0,h^{**})$ and any two initial values $X_0=x,Y_0=y\in\mR^d$, for $k=0,1,2\cdots,\lfloor\frac{T}{h}\rfloor$, the solution generated by the EM scheme (\ref{a-2}) satisfies
$$
\mE|\hat{X}_{kh}-\hat{Y}_{kh}|^2\leq\mE|x-y|^2e^{-\xi_1 hk},
$$
where the constant $\xi_1=\alpha-\beta-(a+b)h^{**}>0$.
\end{lem}
\begin{proof} Applying the method of Lemma \ref{lem-1} yields,
$$
\aligned
&\mE|\hat{X}_{(k+1)h}-\hat{Y}_{(k+1)h}|^2\\=&\mE|\hat{X}_{kh}-\hat{Y}_{kh}|^2+h^2\mE|b(\hat{X}_{kh},\cL_{\hat{X}_{kh}})-b(\hat{Y}_{kh},\cL_{\hat
{Y}_{kh}})|^2\\
&+2h\mE[\langle b(\hat{X}_{kh},\cL_{\hat{X}_{kh}})-b(\hat{Y}_{kh},\cL_{\hat{Y}_{kh}}),\hat{X}_{kh}-\hat{Y}_{kh}\rangle]
+h\mE\|\s(\hat{X}_{kh},\cL_{\hat{X}_{kh}})-\s(\hat{Y}_{kh},\cL_{\hat{X}_{kh}})\|_{\mathrm{HS}}^2\\
\leq&(1-(\alpha-\beta) h+ (a+b)h^2)\mE|\hat{X}_{kh}-\hat{Y}_{kh}|^2:=A_3\mE|\hat{X}_{kh}-\hat{Y}_{kh}|^2. \\
\endaligned
$$
Choose $h^{**}$ sufficiently small such that for all $h\in(0,h^{**})$, $A_3\in(0,1)$. By iteration,
$$
\mE|\hat{X}_{kh}-\hat{Y}_{kh}|^2\leq C_2\mE|x-y|^2,
$$
where $C_2=A_3^k$. Since the fact $a^k< e^{-(1-a)k}$ for any $a\in(0,1)$, the assertion follows.\\
\end{proof}
\begin{lem}\label{lem-3}[Lemma 3.1 in \cite{Zha}] Suppose that the coefficients $b(x,\mu)$ and $\s (x,\mu)$ satisfy $\mathbf{Assumptions}$ $\mathbf{1}$-$\mathbf{3}$ and that $\mE|X_0|^2<\infty$. Then, for any $s,t\in[0,T]$, there exists a constant $C_3$, depending on $T, c_0$, such that
$$
\mE[\sup_{t\in[0,T]}|X_t|^2]\leq C_3(1+\mE|X_0|^2)e^{C_3T},
$$
$$
\mE|X_s-X_t|^2\leq C_3(1+\mE|X_0|^2)(t-s).
$$
\end{lem}

\begin{proof}[The proof of Theorem \ref{thm-1}.]
Combining with (\ref{a-2}), for each $k=0, 1,2,\cdots, \lfloor\frac{T}{h}\rfloor-1$, it holds that
$$
\aligned
&X_{(k+1)h}-\hat{X}_{(k+1)h}\\= &X_{kh}-\hat{X}_{kh}+ \int_{kh}^{(k+1)h}\left(b\left(X_s,\cL_{X_s}\right)-b\left(\hat{X}_{kh},\cL_{\hat{X}_{kh}}\right)\right)\dif s
+\int_{kh}^{(k+1)h}\s(X_{s},\cL_{\hat{X}_{s}})-\s(\hat{X}_{kh},\cL_{\hat{X}_{kh}})\dif W_s.\\
\endaligned
$$
Thus, we have
$$
\aligned
&\mE|X_{(k+1)h}-\hat{X}_{(k+1)h}|^2\\ \leq & \mE|X_{kh}-\hat{X}_{kh}|^2+\mE\left|\int_{kh}^{(k+1)h}\left(b(X_s,\cL_{X_s})-b\left(\hat{X}_{kh},\cL_{\hat{X}_{kh}}\right)\right)\dif s\right|^2\\
&+2\mE\left|\int_{kh}^{(k+1)h}\left\langle b(X_s,\cL_{X_s})-b\left(\hat{X}_{kh},\cL_{\hat{X}_{kh}}\right),X_{kh}-\hat{X}_{kh}\right\rangle\dif s\right|\\&+\mE\left|\int_{kh}^{(k+1)h}\left(\s(X_{s},\cL_{X_{s}})-\s(\hat{X}_{kh},\cL_{\hat{X}_{kh}})\right)\dif W_s\right|^2\\
:=& \mE|X_{kh}-\hat{X}_{kh}|^2+\cA_1+\cA_2+\cA_3.\\
\endaligned
$$
We now estimate each term separately.\\
$\mathbf{Estimate}$ $\mathbf{of}$ $\mathbf{\cA_1}$. From $\mathbf{Assumption}$ $\mathbf{3}$ and Lemma \ref{lem-3}, it follows that
$$
\aligned
\cA_1\leq& h\mE\left(\int_{kh}^{(k+1)h}\left|b(X_s,\cL_{X_s})-b\left(\hat{X}_{kh},\cL_{\hat{X}_{kh}}\right)\right|^2\dif s\right)\\
\leq& h\mE\left(\int_{kh}^{(k+1)h}\left|b(X_s,\cL_{X_s})-b\left(X_{kh},\cL_{X_{kh}}\right)+b\left(X_{kh},\cL_{X_{kh}}\right)-b\left(\hat{X}_{kh},\cL_{\hat{X}_{kh}}\right)\right|^2\dif s\right)\\
\leq &2(a+b)h \int_{kh}^{(k+1)h}\mE|X_s-X_{kh}|^2\dif s+2(a+b)h^2\mE|X_{kh}-\hat{X}_{kh}|^2\\
\leq &2C_3(1+\mE|X_0|^2)(a+b)h \int_{kh}^{(k+1)h}(s-kh)\dif s+2(a+b)h^2\mE|X_{kh}-\hat{X}_{kh}|^2\\
\leq &C_3(1+\mE|X_0|^2)(a+b)h^3+2(a+b)h^2\mE|X_{kh}-\hat{X}_{kh}|^2.\\
\endaligned
$$
$\mathbf{Estimate}$ $\mathbf{of}$ $\mathbf{\cA_2}$. Employing $\mathbf{Assumption}$ $\mathbf{3}$ and the above inequality, yields
$$
\aligned
\cA_2\leq &\mE\int_{kh}^{(k+1)h}\left|b(X_s,\cL_{X_s})-b\left(\hat{X}_{kh},\cL_{\hat{X}_{kh}}\right)\right|^2\dif s+h\mE|X_{kh}-\hat{X}_{kh}|^2\\
\leq &C_3(1+\mE|X_0|^2)(a+b)h^2+(2a+2b+1)h\mE|X_{kh}-\hat{X}_{kh}|^2.\\
\endaligned
$$
$\mathbf{Estimate}$ $\mathbf{of}$ $\mathbf{\cA_3}$. Similarly for $\cA_1$, a further application of $\mathbf{Assumption}$ $\mathbf{3}$ and Lemma 3.3 leads us to
$$
\cA_3\leq C_3(1+\mE|X_0|^2)(a+b)h^2+2(a+b)h\mE|X_{kh}-\hat{X}_{kh}|^2.
$$
Therefore, there exist some positive constants $C_4$ and $C_5$, such that the following estimate holds:
$$
\mE|X_{(k+1)h}-\hat{X}_{(k+1)h}|^2\leq (1+C_4h)\mE|X_{kh}-\hat{X}_{kh}|^2 +C_5h^2.
$$
Applying the discrete-Gronwall's inequality, we have that
$$
\mE|X_{(k+1)h}-\hat{X}_{(k+1)h}|^2\leq C_5h^2\frac{e^{C_4T}-1}{C_4h}:=Ch,
$$
where $C > 0$ is a constant depending on $T, c_0$, $a$, $b$.
Thus, the conclusion is established.
\end{proof}
\subsection{The existence and uniqueness of the numerical invariant measure}

Based on the lemmas above, we can now establish the main theorem concerning the existence and uniqueness of the invariant measure.
\begin{proof}[The proof of Theorem \ref{thm-2}.] We prove the two assertions separately.

\textbf{(1) Existence and uniqueness of the invariant measure for the EM scheme.}

 We first establish the existence of an invariant measure for the EM scheme \eqref{a-2}.  For any initial value $X_0=x\in\mR^d$, denote by $\hat{\mP}_{kh}(x,\cdot)$ the transition kernel of the Markov chain $\{\hat{X}_{kh}\}_{k\geq 0}$. Since $\{\d_x\hat{\mP}_{kh}\}_{k\geq 0}$ is tight, there exists a subsequence converging weakly to $\hat{\pi}_{h}\in\cP(\mR^d)$.

 Using Lemma \ref{lem-2}, defined initial value $Y_0=y\in\mR^d$, we have
\beq\label{1}
\cW_2(\d_x\hat{\mP}_{kh},\d_y\hat{\mP}_{kh})^2\leq \mE|\hat{X}_{kh}-\hat{Y}_{kh}|^2\leq |x-y|^2e^{-\xi_1hk}.
\deq
By the Kolmogorov-Chapman equation and Lemma \ref{lem-1}, for any $k,l>0$, there exists a constant $C_6>0$, such that
$$
\aligned
\cW_2(\d_x\hat{\mP}_{kh},\d_x\hat{\mP}_{(k+l)h})^2= & \cW_2(\d_x\hat{\mP}_{kh},\d_x\hat{\mP}_{kh}\hat{\mP}_{lh})^2\\
\leq &\int_{\mR^d}\cW_2(\d_x\hat{\mP}_{kh},\d_y\hat{\mP}_{kh})^2\hat{\mP}_{lh}(x,\dif y)\\
\leq &\int_{\mR^d}|x-y|^2e^{-\xi_1hk}\hat{\mP}_{lh}(x,\dif y)\\
\leq &2(\mE|X_0|^2+\mE|\hat{X}_{lh}|^2)e^{-\xi_1hk}\leq C_6e^{-\xi_1hk}.
\endaligned
$$
  Thus, in the limit as $l\rightarrow\infty$, it follows that
\beq\label{2}
\cW_2(\d_x\hat{\mP}_{kh},\hat{\pi}_h)^2\leq C_6e^{-\xi_1hk}.
\deq
Letting $k\rightarrow\infty$, we can see that
$$
\cW_2(\d_x\hat{\mP}_{kh},\hat{\pi}_h)^2\rightarrow 0,
$$
which guarantees that $\hat{\pi}_h$ is a unique invariant measure of $\{\hat{X}_{kh}\}$.

In the following, we verify the uniqueness. Suppose that $\hat{\pi}
_h^1\in\cP(\mR^d)$ and $\hat{\pi}_h^2\in\cP(\mR^d)$ are two invariant measures with initial values $x$ and $y$ for the EM scheme \eqref{a-2}. For any $x,y\in\mR^d$ with $x\neq y$, we can obtain that
$$
\cW_2(\hat{\pi}_h^1, \hat{\pi}_h^2)^2\leq \int_{\mR^d\times\mR^d}\cW_2(\d_x\hat{\mP}_{kh},\d_y\hat{\mP}_{kh})^2\nu(\dif x,\dif y).
$$
Applying (\ref{1}) and taking $k\rightarrow+\infty$  yields
$$
\cW_2(\hat{\pi}_h^1, \hat{\pi}_h^2)^2=0.
$$
Thus, the uniqueness of the invariant measure is established.

\textbf{(2) Convergence rate of the numerical invariant measure.}

 According to $\mathbf{(i)}$ of Theorem \ref{thm-2}, there exist positive constants $\xi_1$, $\xi_2$, $C$ such that for all $k\geq0$,
$$
\cW_2(\hat{\pi}_h,\d_{\hat{\mP}_{kh}})^2\leq Ce^{-\xi_1hk},\ \  \cW_2(\pi,\d_{\mP_{kh}})^2\leq  Ce^{-\xi_2hk}.
$$
where $\mP_{kh}$ denotes the transition kernel of the continuous-time solution $X_{kh}$.

For any fixed step size $h\in(0,h^{\sharp})$, one can choose $k$ sufficiently large such that
$$
\cW_2(\hat{\pi}_h,\d_{\hat{\mP}_{kh}})^2\leq Ch,\ \  \cW_2(\pi,\d_{\mP_{kh}})^2\leq  Ch.
$$
Then, for the same $k$, Theorem \ref{thm-1} provides
$$\cW_2(\d_{\hat{\mP}_{kh}},\d_{\mP_{kh}})^2\leq \mE|X_{kh}-\hat{X}_{kh}|^2\leq C h.
$$
Finally, by the triangle inequality for the Wasserstein distance,
$$
\cW_2(\hat{\pi}_h,\pi)^2\leq \cW_2(\hat{\pi}_h,\d_{\hat{\mP}_{kh}})^2+\cW_2(\d_{\hat{\mP}_{kh}},\d_{\mP_{kh}})^2+\cW_2(\pi,\d_{\mP_{kh}})^2\leq Ch.
$$
Thus, $\cW_2(\hat{\pi}_h,\pi)\preceq h^{1/2}$, completing the proof of Theorem \ref{thm-2}.
\end{proof}
\section{The numerical invariant measure for the particle system}
The purpose of this section is to analyze the convergence of the particle system and to establish the existence and uniqueness of an invariant measure for the numerical solution.
\subsection{The properties of numerical invariant measures for particle systems}
\begin{lem}\label{lem-5} Suppose that $\mathbf{Assumptions}$ $\mathbf{2}$-$\mathbf{3}$ hold and that $\mE|X_0|^2<\infty$. For any step size $h\in(0,h^*)$ and for some constant $C_1$, the EM scheme (\ref{c-2}) satisfy
$$
\sup_{k=1,2,\cdots,\lfloor\frac{T}{h}\rfloor,j=1,2,\cdots,N}\mE|\hat{X}_{kh}^{j,h}|^2\leq C_1,
$$
where $h^*$ is a step size in $(0,1)$.
\end{lem}
\begin{lem}\label{lem-4} Under $\mathbf{Assumptions}$ $\mathbf{2}$-$\mathbf{4}$, there exists $h^{**}\in(0,1)$  with the following property: for any step size  $h\in(0,h^{**})$ and any two initial values $\hat{X}_{0}^{N,j,h}=x$, $\hat{Y}_{0}^{N,j,h}=y$, the EM scheme $\{\hat{X}_{kh}^{N,j,h}\}
$ and $\{\hat{Y}_{kh}^{N,j,h}\}$ defined by scheme (\ref{b-2}) satisfy, for $k=1,2,\cdots,\lfloor\frac{T}{h}\rfloor$ and $j=1,2,\cdots,N$,
$$
\mE\left|\hat{X}_{kh}^{N,j,h}-\hat{Y}_{kh}^{N,j,h}\right|^2\leq\mE|x-y|^2e^{-\xi_1 hk},
$$
where $\xi_1=\alpha-\beta-(a+b)h^{**}>0$.
\end{lem}
\begin{proof} Following the approach of Lemma \ref{lem-2} and by $\mathbf{Assumptions}$ $\mathbf{2}$-$\mathbf{3}$, one can see that
$$
\aligned
&\mE\left|\hat{X}_{(k+1)h}^{N,j,h}-\hat{Y}_{(k+1)h}^{N,j,h}\right|^2\\
\leq &(1-\a h+ah^2)\mE\left|\hat{X}_{kh}^{N,j,h}-\hat{Y}_{kh}^{N,j,h}\right|^2+(\b h+bh^2)\mE\left[\cW_2\left(\frac{1}{N}\sum_{i=1}^N\d_{\hat{X}_{kh}^{N,i,h}},\frac{1}{N}\sum_{i=1}^N\d_{\hat{Y}_{kh}^{N,i,h}}\right)^2\right].
\endaligned
$$
A key property of the Wasserstein distance for empirical measures is that
$$
\mE\left[\cW_2\left(\frac{1}{N}\sum_{i=1}^N\d_{\hat{X}_{kh}^{N,i,h}},\frac{1}{N}\sum_{i=1}^N\d_{\hat{Y}_{kh}^{N,i,h}}\right)^2\right]\leq \frac{1}{N}\sum_{i=1}^N\mE\left(\hat{X}_{kh}^{N,i,h}-\hat{Y}_{kh}^{N,i,h}\right)^2.
$$
Using symmetry, summing over $j$ yields:
$$
\mE\left|\hat{X}_{(k+1)h}^{N,j,h}-\hat{Y}_{(k+1)h}^{N,j,h}\right|^2\leq (1-(\a-\b) h+(a+b)h^2)\mE\left|\hat{X}_{kh}^{N,j,h}-\hat{Y}_{kh}^{N,j,h}\right|^2.
$$
The remainder of the proof follows exactly as in Lemma \ref{lem-2}, leading to the desired exponential decay estimate.
\end{proof}

\subsection{The existence and uniqueness of the numerical invariant measure for particle systems}
\begin{proof}[Proofs of Theorem \ref{thm-3}]By a straightforward calculation, we obtain
$$
\mE|\hat{X}_{(k+1)h}^{N,j,h}-\hat{X}_{(k+1)h}^{j,h}|^2
\leq \mE|\hat{X}_{kh}^{N,j,h}-\hat{X}_{kh}^{j,h}|^2+\cC_1+\cC_2,
$$
where
$$
\cC_1:=h^2\mE\left[\left|b\left(\hat{X}_{kh}^{N,j,h},\frac{1}{N}\sum_{i=1}^N\d_{\hat{X}_{kh}^{N,i,h}}
\right)-b\left(\hat{X}_{kh}^{j,h},
\cL_{\hat{X}_{kh}^{j,h}}\right)\right|^2\right],
$$
$$
\aligned
\cC_2:=&
2h\mE\left\langle b\left(\hat{X}_{kh}^{N,j,h},\frac{1}{N}\sum_{i=1}^N\d_{\hat{X}_{kh}^{N,i,h}}
\right)-b\left(\hat{X}_{kh}^{j,h},
\cL_{\hat{X}_{kh}^{j,h}}\right), \hat{X}_{kh}^{N,j,h}-\hat{X}_{kh}^{j,h} \right\rangle\\
&+h\mE\left[\left\|\s\left(\hat{X}_{kh}^{N,j,h},\frac{1}{N}\sum_{i=1}^N\d_{\hat{X}_{kh}^{N,i,h}}\right)-\s\left(\hat{X}_{kh}^{j,h},\cL_{\hat{X}_{kh}^{j,h}}\right)\right\|^2_{\mathrm{HS}}\right].\\
\endaligned
$$
$\mathbf{Estimate}$ $\mathbf{of}$ $\mathbf{\cC_1}$. It can be readily verified that
$$
\aligned
\cC_1\leq &2h^2\mE\left[\left|b\left(\hat{X}_{kh}^{N,j,h},\frac{1}{N}\sum_{i=1}^N\d_{\hat{X}_{kh}^{N,i,h}}\right)-b\left(\hat{X}_{kh}^{j,h},
\frac{1}{N}\sum_{i=1}^N\d_{\hat{X}_{kh}^{i,h}}\right)\right|^2\right]\\
&+2h^2\mE\left[\left|b\left(\hat{X}_{kh}^{j,h},
\frac{1}{N}\sum_{i=1}^N\d_{\hat{X}_{kh}^{i,h}}\right)-b\left(\hat{X}_{kh}^{j,h},
\cL_{\hat{X}_{kh}^{j,h}}\right)\right|^2\right]:=2h^2(\cC_{1,1}+\cC_{1,2}).
\endaligned
$$

For $\cC_{1,1}$, by $\mathbf{Assumption}$ $\mathbf{3}$ and the Cauchy-Schwarz inequality,
$$
\aligned
\cC_{1,1}=&\frac{1}{N^2}\mE\left[\sum_{i=1}^N\left(b\left(\hat{X}_{kh}^{N,j,h},\d_{\hat{X}_{kh}^{N,i,h}}\right)-b
\left(\hat{X}_{kh}^{j,h},\d_{\hat{X}_{kh}^{i,h}}\right)\right)\right]^2\\
\leq &\frac{1}{N}\sum_{i=1}^N\mE\left[\left|b\left(\hat{X}_{kh}^{N,j,h},\d_{\hat{X}_{kh}^{N,i,h}}\right)-b
\left(\hat{X}_{kh}^{j,h},\d_{\hat{X}_{kh}^{i,h}}\right)\right|^2\right]\leq (a+b) \mE|\hat{X}_{kh}^{N,j,h}-\hat{X}_{kh}^{j,h}|^2,\\
\endaligned
$$
where the last inequality holds by summing over $j$ on both sides, using the symmetry, and removing the summation symbol.

For $\cC_{1,2}$,
denoting $Z_s^{h,i,j}=b\left(\hat{X}_s^{j,h},\d_{\hat{X}_s^{i,h}}\right)-b\left(\hat{X}_s^{j,h},\cL_{\hat{X}_s^{j,h}}\right)$, and by the independence of $\{\hat{X}_{\cdot}^{j,h}\}_{j=1}^N$,
for any $i\neq j\neq l$, we have $\mE\langle Z_{kh}^{h,i,j},Z_{kh}^{h,l,j}\rangle=0$. Therefore, there exists a positive constant $\tilde{C}_1$, such that
$$
\aligned
\cC_{1,2}=&\frac{1}{N^2}\mE\left[\left|\sum_{i=1}^NZ_{kh}^{h,i,j}\right|^2\right]=\frac{1}{N^2}\sum_{j,l=1}^N\mE\langle Z_{kh}^{h,i,j},Z_{kh}^{h,l,j}\rangle \\
\leq & \frac{1}{N^2} \left(2\sum_{l=1}^N\mE\langle Z_{kh}^{h,j,j},Z_{kh}^{h,l,j}\rangle+\sum_{i=1}^N\mE|Z_{kh}^{h,i,j}|^2\right)
\leq \frac{2b^2\mE|\hat{X}_{kh}^{j,h}|^2}{N}:=\tilde{C}_1\frac{1}{N},\\
\endaligned
$$
where the last inequality holds using Lemma \ref{lem-5}. Hence, combining the estimates for $\cC_{1,1}$ and $\cC_{1,2}$, we obtain
$$
\cC_1\leq 2(a+b)h^2\mE|\hat{X}_{kh}^{N,j,h}-\hat{X}_{kh}^{j,h}|^2+\tilde{C}_1\frac{2h^2}{N}.
$$
$\mathbf{Estimate}$ $\mathbf{of}$ $\mathbf{\cC_2}$.
Applying the same method shows that
\beq\label{4}
\cC_{2}\leq -\a h\mE|\hat{X}_{kh}^{N,j,h}-\hat{X}_{kh}^{j,h}|^2+\b h\mE\left[\cW_2\left(\frac{1}{N}\sum_{i=1}^N\d_{\hat{X}_{kh}^{N,i,h}},\cL_{\hat{X}_{kh}^{j,h}}\right)^2\right],
\deq
where
$$
\aligned
&\mE\left[\cW_2\left(\frac{1}{N}\sum_{i=1}^N\d_{\hat{X}_{kh}^{N,i,h}},\cL_{\hat{X}_{kh}^{j,h}}\right)^2\right]\\ \leq & \mE\left[\cW_2\left(\frac{1}{N}\sum_{i=1}^N\d_{\hat{X}_{kh}^{N,i,h}},\frac{1}{N}\sum_{i=1}^N\d_{\hat{X}_{kh}^{i,h}}\right)^2\right]+
\mE\left[\cW_2\left(\frac{1}{N}\sum_{i=1}^N\d_{\hat{X}_{kh}^{i,h}},\cL_{\hat{X}_{kh}^{j,h}}\right)^2\right].\\
:=&\cC_{2,1}+\cC_{2,2}.\\
\endaligned
$$
For the term $\cC_{2,1}$, we have
$$
\cC_{2,1}\leq \frac{1}{N}\sum_{i=1}^{N}\mE\left|\hat{X}_{kh}^{N,i,h}-\hat{X}_{kh}^{i,h}\right|^2.
$$
For the term $\cC_{2,2}$, using a method similar to $\cC_{1,2}$, for some constant $\tilde{C}_2>0$, it follows that
$$
\cC_{2,2}\leq \mE\left|\frac{1}{N}\sum_{i=1}^N\hat{X}_{kh}^{i,h}-\hat{X}_{kh}^{j,h}\right|^2\leq \frac{1}{N^2}\sum_{i,l=1}^N\mE\langle\hat{X}_{kh}^{i,h}-\hat{X}_{kh}^{j,h},\hat{X}_{kh}^{l,h}-\hat{X}_{kh}^{j,h}\rangle\leq \frac{2\mE|\hat{X}_{kh}^{j,h}|^2}{N}:=\tilde{C}_2\frac{1}{N}.
$$
Inserting the above relations into (\ref{4}), summing over $j$ on both sides, using the symmetry, and removing the summation symbol, for a constant $\tilde{C}_3$, then
$$
\cC_{2}\leq -(\a-\b)h\mE|\hat{X}_{kh}^{N,j,h}-\hat{X}_{kh}^{j,h}|^2+\tilde{C}_3\frac{h}{N}.
$$
Thus, there exists a constant $\tilde{C}_4$ such that
$$
\sup_{i=1,2,\cdots,N}\mE|\hat{X}_{(k+1)h}^{N,j,h}-\hat{X}_{(k+1)h}^{j,h}|^2\leq (1-(\a-\b)h+(a+b)h^2) \mE|\hat{X}_{kh}^{N,j,h}-\hat{X}_{kh}^{j,h}|^2+\tilde{C}_4\frac{h}{N} .
$$
Using the fact that
$$
\sup_{j=1,2,\cdots,N}\mE\left |X_{0}^{N,j,h}-X_0^{j,h}\right|^2=0.
$$
Thus, choosing a sufficient small $h^{\sharp}$, for $h\in(0,h^{\sharp})$, it follows that
$$
\sup_{j=1,2,\cdots,N}\mE\left[|\hat{X}_{kh}^{N,j,h}-\hat{X}_{kh}^{j,h}|^2\right]\preceq \frac{h}{N}.
$$
This completes the proof.
\end{proof}
\begin{proof}[The proof of Theorem \ref{thm-4}.] We prove the two assertions separately.

\textbf{(1) Existence and uniqueness of the invariant measure for the particle system.}

 By applying the method of Theorem \ref{thm-2}, we find that $\hat{\pi}_h^N$ is the unique invariant measure for the semigroup ${\delta_x\hat{\mP}_{kh}^{N,j}}$, where $\hat{\mP}_{kh}^{N,j}(x,B)$ is the transition probability kernel of $\hat{X}_{kh}^{N,j,h}$, for $j=1,2,\cdots,N$ and $k=0,1,2,\cdots, \lfloor\frac{T}{h}\rfloor$.

 \textbf{(2) Approximation error between the particle system's invariant measure and the true invariant measure.}

 For any $j=1,2,\cdots,N$, let $\hat{\mP}_{kh}^j(x,B)$ be the transition probability kernel of the non-interacting particle system $\{\hat{X}_{kh}^{j,h}\}$ defined by \eqref{c-2}. Then,
by the triangle inequality, we can get
$$
\cW_2(\hat{\pi}_h^N,\pi)^2\leq \cW_2(\hat{\pi}_h^N,\d_{\hat{\mP}_{kh}^{N,j}})^2+\cW_2(\d_{\hat{\mP}_{kh}^{N,j}},\d_{\hat{\mP}_{kh}^j})^2+\cW_2(\pi,\d_{\hat{\mP}_{kh}^j})^2.
$$
According to $\mathbf{(i)}$ in Theorem \ref{thm-2} and Theorem \ref{thm-4}, the convergence of $\hat{\mP}_{kh}^{N,j}$ to $\hat{\pi}_h^N$ and $\hat{\mP}_{kh}^j$ to $\pi$ as $h\rightarrow 0$, yields
$$
\cW_2(\hat{\pi}_h^N,\d_{\hat{\mP}_{kh}^{N,j}})^2\leq Ce^{-\xi_1hk},\ \  \cW_2(\pi,\d_{\hat{\mP}_{kh}^j})^2\leq  Ce^{-\xi_2hk},
$$
where $\xi_1$ and $\xi_2$ are positive constants.

For any step size  $h\in(0,h^{\sharp})$, one can choose $k$ sufficiently large, we arrive at
$$
\cW_2(\hat{\pi}_h^N,\d_{\hat{\mP}_{kh}^{N,j}})^2\preceq h,\ \  \cW_2(\pi,\d_{\hat{\mP}_{kh}^j})^2\preceq h.
$$
Furthermore, for the same $k$, using Theorem \ref{thm-3}, it is straightforward to obtain that
$$\cW_2(\d_{\hat{\mP}_{kh}^{N,j}},\d_{\hat{\mP}_{kh}^j})^2\leq \mE|X_{kh}^{N,j,h}-\hat{X}_{kh}^{j,h}|^2\preceq \frac{h}{N},
$$
which is equivalent to the statement $\cW_2(\d_{\hat{\mP}_{kh}^{N,j}},\d_{\hat{\mP}_{kh}^j})\preceq h^{1/2}$ in the limit $N\rightarrow \infty$.
This completes the proof of Theorem \ref{thm-4}.
\end{proof}
\section{Numerical examples}
This section presents numerical simulations to illustrate the theoretical results established in the previous sections. We consider a one-dimensional linear McKean-Vlasov SDE that satisfies the monotonicity condition. Meanwhile, we examine the behavior of the EM scheme, the interacting particle system, and their invariant measures.
\begin{exa} Consider the following linear McKean-Vlasov SDE:
$$
\dif X_t=(-\lambda X_t+\theta \mE X_t)\dif t+\s_0\dif W_t.
$$
with parameters $\lambda=1.2$, $\theta=0.4$ and $\s_0=1$. It is straightforward to confirm that  this equation satisfies $\mathbf{Assumptions}$ $\mathbf{1}$-$\mathbf{4}$, ensuring the existence and uniqueness of an invariant measure for both the continuous-time process and its EM scheme \eqref{a-2}.

\textbf{(1) Existence and uniqueness of the invariant measure.}

To verify the existence of an invariant measure, we simulate $10000$ sample paths with the initial value of $X_0 = 6$ and step size of $h = 0.01$. The left panel of Figure 1 plots the resulting empirical density functions at several time points. While the densities at $t = 0.1$, $0.3$, and $0.5$ differ significantly, those at $t = 4$ and $t = 8$ align closely. This supports the existence of an invariant measure.

To demonstrate the uniqueness,
the right panel compares empirical densities at time $t = 30$ for three different initial values: $X_0=-6, 6, 16$. Their remarkable similarity provides strong evidence for the uniqueness of the invariant measure.
\begin{figure}[htbp]\label{F-1}
\centering
\includegraphics[scale=0.5]{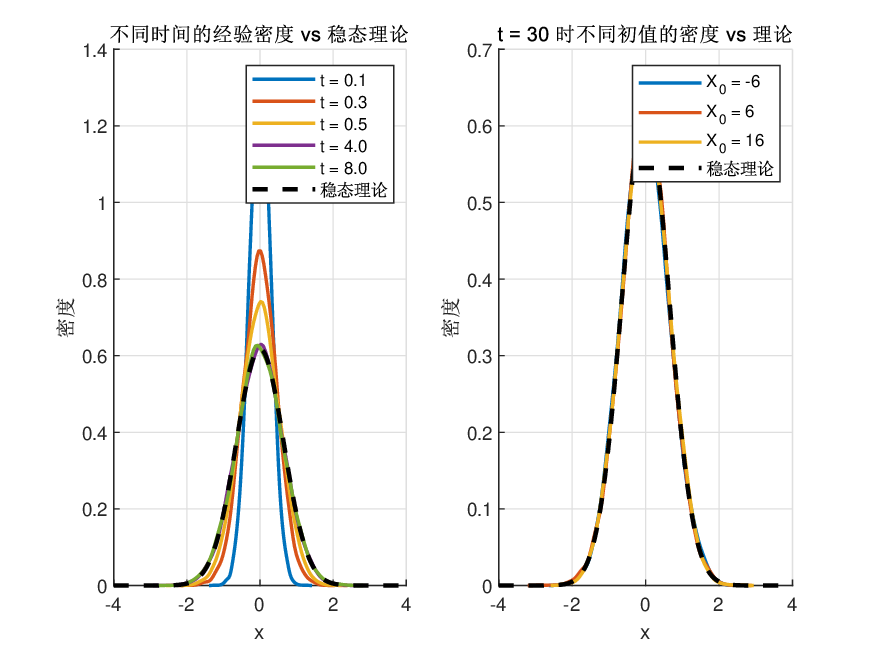}
\caption{Left: Empirical density functions at different time points. Right: Empirical density functions at t=30 with different initial values.}
\label{figure}
\end{figure}

\textbf{(2) Strong convergence of the EM scheme.}

Figure 2 illustrates the strong convergence rate of the EM scheme, confirming the theoretical result of Theorem \ref{thm-1}. A set of step sizes $ h = 0.04, 0.02, 0.01, 0.005, 0.0025$, along with a refined step size  $h_{ref}=e^{-4}$ are selected. Pathwise consistency is enforced through a multiscale algorithm based on coarse-fine grid comparisons. The log-log plot exhibits a slope of approximately $1/2$, consistent with the predicted $O(h^{1/2})$ convergence rate.

\begin{figure}[htbp]\label{F-2}
\centering
\includegraphics[scale=0.4]{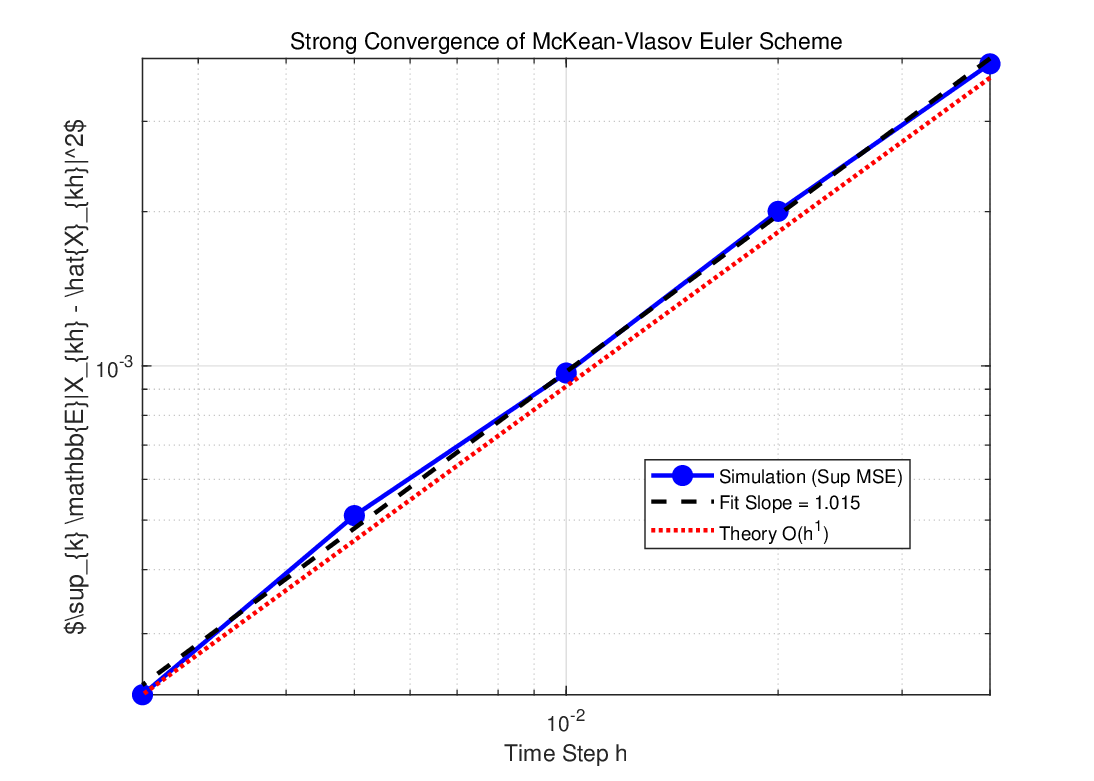}
\caption{Strong convergence of the Euler-Maruyama scheme.}
\label{figure}
\end{figure}

\textbf{(3) Convergence of the particle system.}

We provide numerical evidence for the convergence rate predicted in Theorem \ref{thm-3} which quantifies the relationship between the interacting particle system and its mean-field limit.

Figure 3 displays two complementary experiments: the convergence with respect to the number of particles
$N$ (Chaos error vs N, left panel) and the convergence with respect to the time step $h$ (Discretization error vs h, right panel).

In the left panel, the step size is fixed at
$h=0.001$, chosen sufficiently small to ensure that the discretization error is dominated by the chaos error. Particle counts
$N=50, 100, 200, 400, 800, 1600$
are selected in geometric progression to clearly exhibit the power-law scaling in log-log coordinates. The objective is to verify the theoretical $1/N$ decay, which corresponds to a line of slope $-1$ in the logarithmic plot.

In the right panel, the particle number is fixed at
$N=100$
 to isolate the discretization error, and the step sizes
$h=0.04,0.02,0.01,0.005,0.0025$
 are taken in a halving sequence. A reference solution computed with a refine step size
$h_{ref}=e^{-4}$ is used to ensure numerical accuracy. This setup allows us to confirm the theoretical
$h$-scaling, which should appear as a straight line of slope
$+1$ under logarithmic scaling.
\begin{figure}[htbp]\label{F-3}
\centering
\includegraphics[scale=0.4]{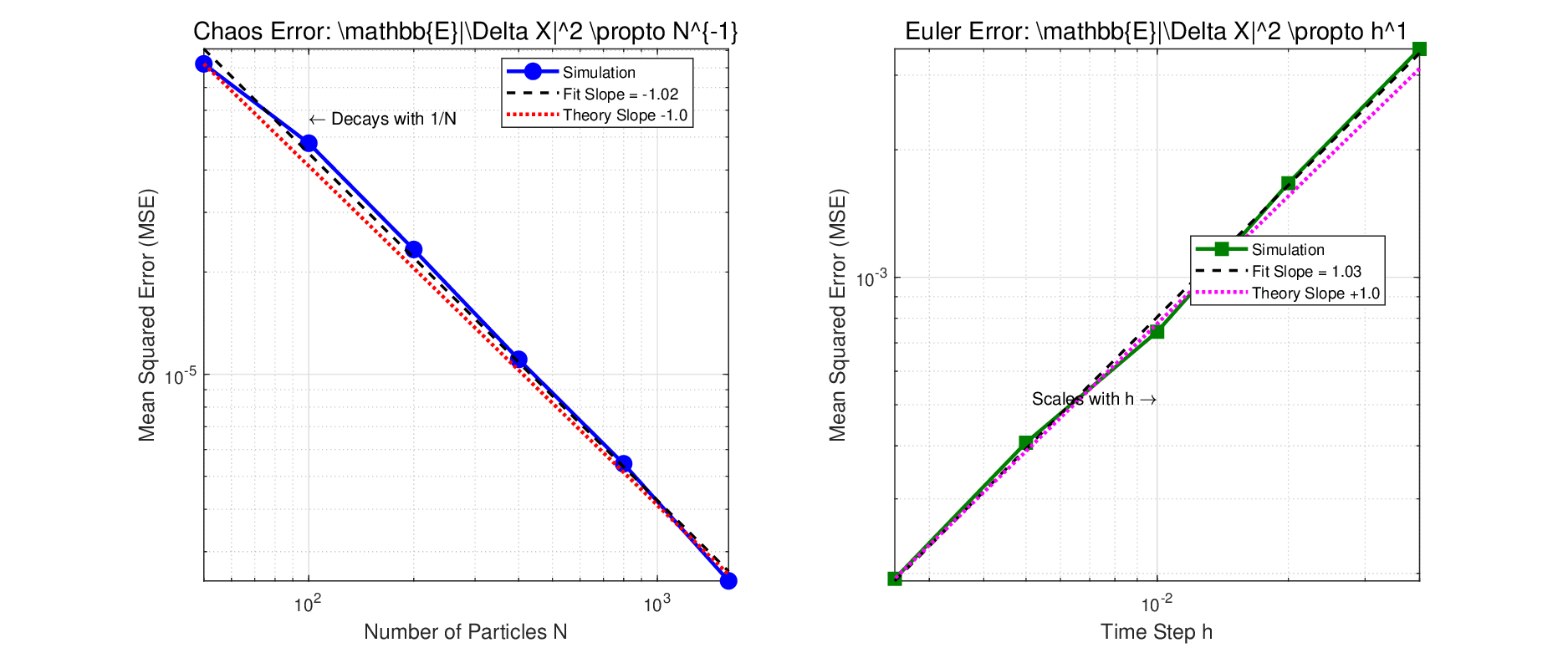}
\caption{Left: Chaos error vs N. Right: Discretization error vs h.}
\label{figure}
\end{figure}

\textbf{(4) Convergence to the invariant measure.}

In Figure 4, a simulation of the McKean-Vlasov SDEs is performed with $N = 1000$ particles, comparing evolutions from initial value $-3$ and $3$ at time points $t = 0, 0.2\tau, 0.5\tau$ (where $\tau$ is the characteristic relaxation time). At the final time, the empirical distributions from the two different initial conditions are indistinguishable, confirming that the invariant measure exists and is unique.
\begin{figure}[htbp]\label{F-4}
\centering
\includegraphics[scale=0.5]{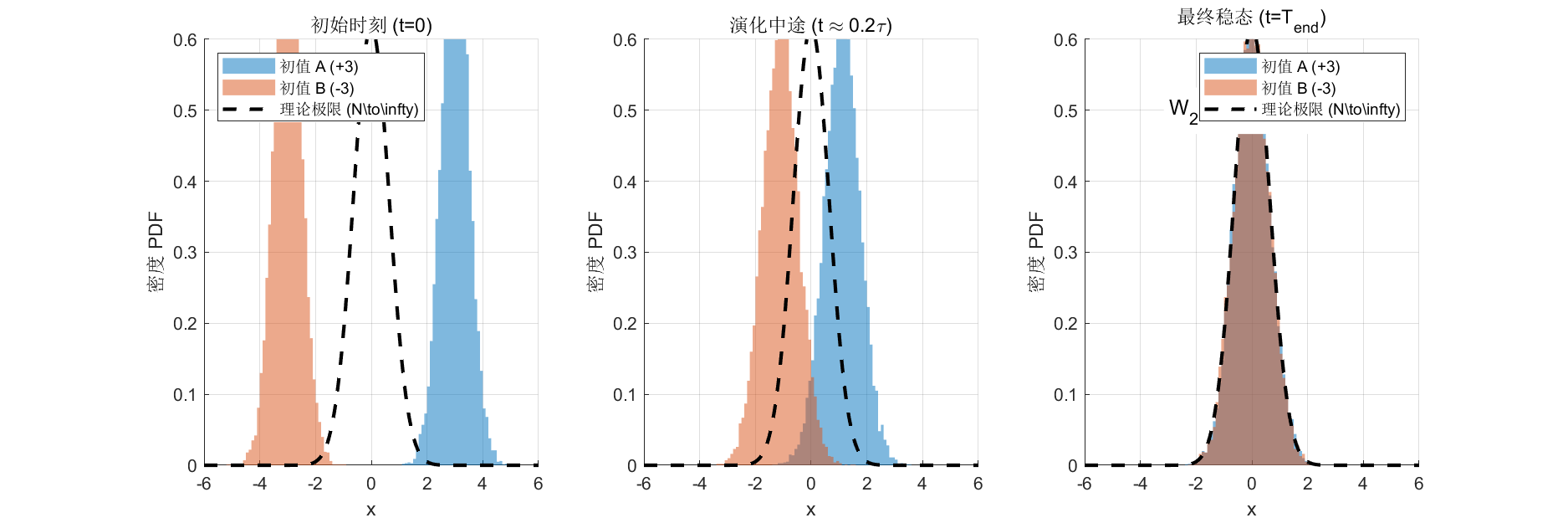}
\caption{Evolution of the particle system toward the invariant measure. Left: Initial configuration. Center: Intermediate stage.  Right: Steady state.}
\label{figure}
\end{figure}
\end{exa}
\section{Data Availability Statement}
All data, models, and code generated or used during the study appear in the submitted article.



\begin{thebibliography}{99}
\bibitem{BSY} J. Bao, M. Scheutzow, C. Yuan, Existence of invariant probability measures for functional McKean-Vlasov SDEs. {\it Electron. J. Probab.}, 27 (2022), Paper No. 43, 14 pp.
\bibitem{CD} W. Cao, K. Du, Empirical approximation to invariant measures of non-degenerate McKean-Vlasov dynamics. {\it Electron. J. Probab.}, 30 (2025), Paper No. 41, 22 pp.
\bibitem{DJL} K. Du, Y. Jiang, J. Li, Empirical approximation to invariant measures for McKean-Vlasov processes: mean-field interaction vs self-interaction. {\it Bernoulli}, 29 (2023), no. 3, 2492-2518.
 \bibitem{HRW} Z. Hao, C. Ren, M. Wu, Supercritical McKean-Vlasov SDE driven by cylindrical $\a$-stable process. {\it Avaiable at arXiv:2410.18611.}
 \bibitem{HWY} X. Huang, W. Shen, F.-F. Yang, Weak solution and invariant probability measure for McKean-Vlasov SDEs with integrable drifts. {\it J. Math. Anal. Appl.}, 537 (2024), no. 2, Paper No. 128318, 15 pp.
\bibitem{K} M. Kac, Foundations of kinetic theory. Proceedings of the Third Berkeley Symposium on Mathematical Statistics and Probability, 1954-1955, vol. III, pp. 171-197, Univ. California Press, Berkeley-Los Angeles, Calif., 1956.
\bibitem{LMW} W. Liu, X. Mao, Y. Wu,
The backward Euler-Maruyama method for invariant measures of stochastic differential equations with super-linear coefficients. {\it
Appl. Numer. Math.}, 184 (2023), 137-150.
\bibitem{M} H. P., Jr. McKean, A class of Markov processes associated with nonlinear parabolic equations.
{\it Proc. Nat. Acad. Sci. U.S.A.}, 56 (1966), 1907-1911.
\bibitem{P} G. A. Pavliotis, Stochastic processes and applications. Diffusion processes, the Fokker-Planck and Langevin equations. {\it Texts Appl. Math.}, 60. Springer, New York, 2014. xiv+339 pp.
  \bibitem{RF} C. Ren\'e, D. Francois, Probabilistic theory of mean field games with applications. I.  Mean field FBSDEs, control, and games. {\it  Probab. Theory Stoch. Model.}, 83. Springer, Cham, 2018. xxv+713 pp.
\bibitem{RF2} C. Ren\'e, D. Francois, Probabilistic theory of mean field games with applications. II. Mean field games with common noise and master equations. {\it  Probab. Theory Stoch. Model.}, 84. Springer, Cham, 2018. xxiv+697 pp.
\bibitem{S} A.-S. Sznitman, Topics in propagation of chaos. \'Ecole $\mathrm{d}'$\'Et\'e de Probabilit\'es de Saint-Flour XIX-1989, 165-251, {\it Lecture Notes in Math., 1464}, Springer, Berlin, 1991.
\bibitem{SKP} L. Sharrock, N. Kantas, P. Parpas, G. A. Pavliotis,
Online parameter estimation for the McKean-Vlasov stochastic differential equation.
{\it Stochastic Process. Appl.}, 162 (2023), 481-546.
\bibitem{W} F.-Y. Wang, Distribution dependent SDEs for Landau type equations. {\it Stochastic Process. Appl.}, 128 (2018), no. 2, 595-621.
\bibitem{W1} F.-Y. Wang,
Exponential ergodicity for singular reflecting McKean-Vlasov SDEs.
{\it Stoch. Process. Appl.}, 160 (2023), pp. 265-293.
\bibitem{We}L. Weng, W. Liu,
Invariant measures of the Milstein method for stochastic differential equations with commutative noise. {\it
Appl. Math. Comput.}, 358 (2019), 169-176.
\bibitem{Zh} S.-Q. Zhang.
Existence and non-uniqueness of stationary distributions for distribution dependent SDEs. {\it
Electron. J. Probab.}, 28 (2023), Paper No. 93, 34 pp.
\bibitem{Zha} X. Zhang, A discretized version of Krylov's estimate and its applications. {\it Electron. J. Probab.}, 24 (2019), Paper No. 131, 17 pp.
\end{thebibliography}
\end{document}